\documentclass[11pt]{article}

\usepackage{amsthm,amsmath,amssymb,graphics,graphicx,hyperref,epstopdf}

\title{Balanced Non-Transitive Dice}
\author{Alex Schaefer\footnote{Binghamton University; aschaef3@binghamton.edu}  \hspace{0pt}  \& Jay Schweig\footnote{Oklahoma State University; jay.schweig@okstate.edu}}
\newtheorem{theorem}{Theorem}[section]

\newtheorem{corollary}[theorem]{Corollary}
\newtheorem{lemma}[theorem]{Lemma}
\newtheorem{definition}[theorem]{Definition}

\newtheorem{question}[theorem]{Question}
\newtheorem{example}[theorem]{Example}
\newtheorem{algorithm}[theorem]{Algorithm}

\newcommand{\ssection}[1]{%
  \section[#1]{\centering\normalfont\scshape #1}}
\newcommand{\ssubsection}[1]{%
  \subsection[#1]{\raggedright\normalfont\itshape #1}}

\begin{document} 

\maketitle

\begin{abstract}
We study triples of labeled dice in which the relation ``is a better die than'' is non-transitive.  Focusing on such triples with an additional symmetry we call ``balance,'' we prove that such triples of $n$-sided dice exist for all $n \geq 3$.  We then examine the sums of the labels of such dice, and use these results to construct an $O(n^2)$ algorithm for verifying whether or not a triple of $n$-sided dice is balanced and non-transitive.  Finally, we consider generalizations to larger sets of dice.  
\end{abstract}

\ssection{Introduction}

Suppose we play the following game with the three six-sided dice in Figure \ref{path}:  You choose a die, and then I choose a die (based on your choice).  We roll our dice, and the player whose die shows a higher number wins.  \\

A closer look at the dice in Figure \ref{path} reveals that, in the long run, I will have an advantage in this game:  Whichever die you choose, I will choose the one immediately to its left (and I will choose die C if you choose die A).  In any case, the probability of my die beating yours is $19/36 > 1/2$.  \\

\begin{figure}[htp]
\centering
\includegraphics[height = 1.05in]{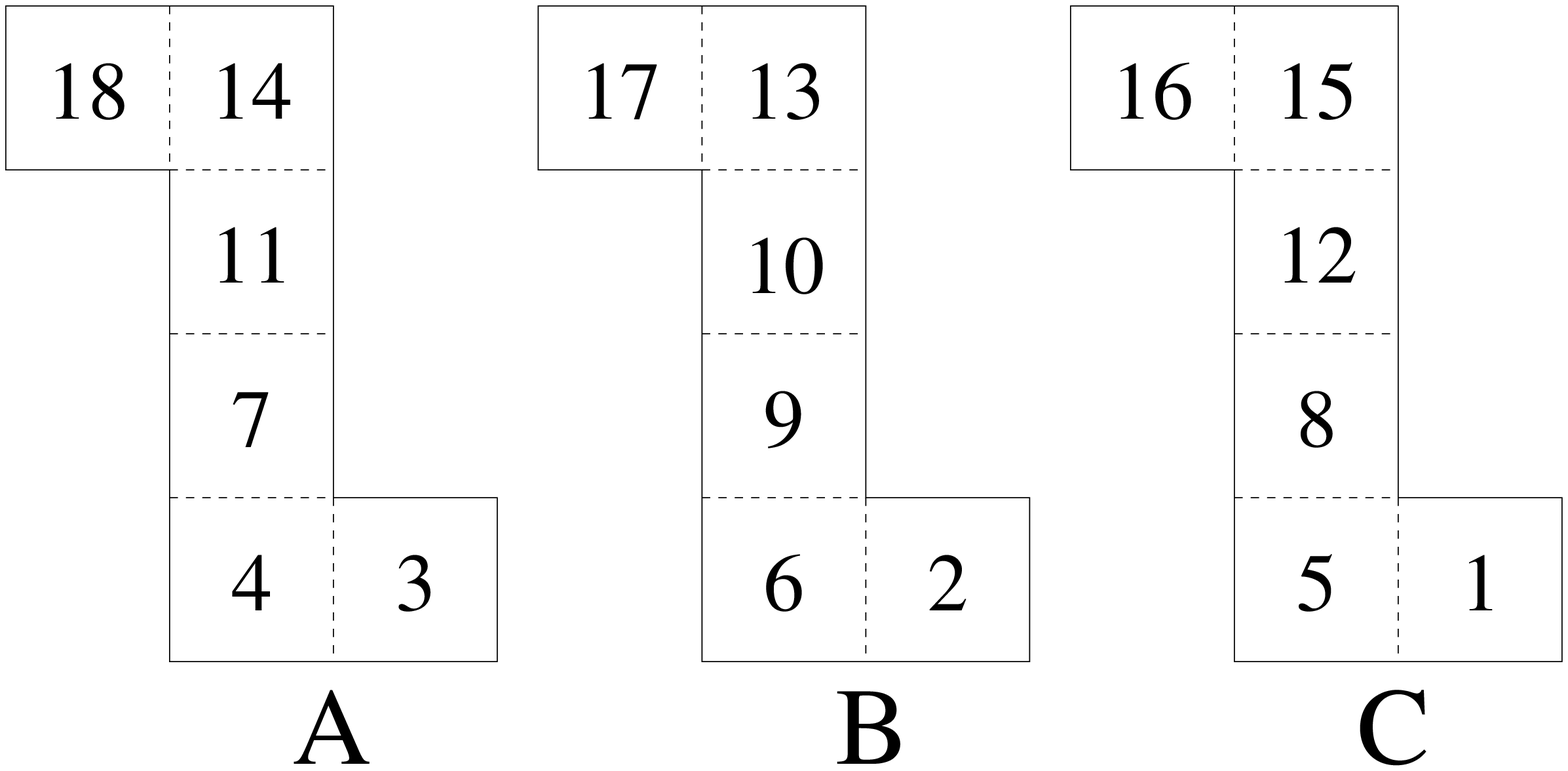}
\caption{A set of balanced non-transitive 6-sided dice.}\label{path}
\end{figure}

This is a case of the phenomenon of \emph{non-transitive} dice, first introduced by Martin Gardner in \cite{MGntD}, and further explored in \cite{MGnt}, \cite{ntD1}, and \cite{FFF}.  \\

We formally define a triple of dice as follows: Fix an integer $n > 0$.  For our purposes, a set of $n$-sided \emph{dice} is a collection of three pairwise-disjoint sets $A,B,$ and $C$ with $|A| = |B| = |C| = n$ and $A \cup B \cup C = [3n]$ (here and throughout, $[n] = \{1, 2, \ldots, n\}$).  We think of die $A$ as being labeled with the elements of $A$, and so on.  Each die is fair, in that the probability of rolling any one of its numbers is $1/n$.  We also write $P(A\succ B)$ to indicate the probability that, upon rolling both $A$ and $B$, the number rolled on $A$ exceeds that on $B$.

\begin{definition}
A set of dice is \emph{non-transitive} if each of $P( A \succ B), P(B \succ C),$ and $P(C \succ A)$ exceeds $1/2$.  That is, the relation ``is a better die than'' is non-transitive.  
\end{definition}

In this paper we (mostly) examine non-transitive sets of dice, but we introduce a new property as well.

\begin{definition}
A set of dice is \emph{balanced} if $P(A \succ B) = P(B \succ C) = P(C \succ A)$.   
\end{definition}

Note that the set of dice in Figure \ref{path} is balanced, as $P(A \succ B) = P(B \succ C) = P(C \succ A) = 19/36$.  \\

In Theorem \ref{thm}, we show that non-transitive balanced sets of $n$-sided dice exist for all $n \geq 3$.  Surprisingly, this is also the first proof that non-transitive sets of $n$-sided dice exist for all $n \geq 3$.  We then prove in Theorem \ref{facesum} that a set of dice is balanced (but not necessarily non-transitive) if and only if the face-sums of the dice are equal (the \emph{face-sum} of a die is simply the sum of the numbers with which it is labeled).  This then yields an $O(n^2)$ algorithm for determining if a given triple of $n$-sided dice is non-transitive and balanced.  Finally, we consider generalizations to sets of four dice.  

\ssection{Balanced dice}

The main goal in this section is to prove the following.

\begin{theorem}\label{thm}
For any $n \geq 3$, there exists a non-transitive set of balanced $n$-sided dice.
\end{theorem}

First, we need some machinery.  Fix $n > 0$.  For our purposes, a \emph{word} $\sigma$ is a sequence of $3n$ letters where each letter is either an $a, b,$ or $c$, and each of $a, b,$ and $c$ appears $n$ times.  

\begin{definition}
If $D$ is a set of $n$-sided dice, define a word $\sigma(D)$ by the following rule: the $i^\text{th}$ letter of $\sigma(D)$ corresponds to the die on which the number $i$ labels a side.  
\end{definition}

Now let $\sigma = s_1 s_2 \cdots s_{3n}$ be a word.  We define a function $q^+_\sigma$ on the letters of $\sigma$ as follows. 
\begin{displaymath}
   q^+_\sigma(s_i) = \left\{
     \begin{array}{lr}
       |\{j < i : s_j = b \}|  & :  s_i = a\\
 |\{j < i : s_j = c \}|  & :  s_i = b\\     
  |\{j < i : s_j = a \}|  & :  s_i = c 
 \end{array}
   \right.
  \end{displaymath} 
Similarly, define a function $q^-_\sigma$ by 
\begin{displaymath}
   q^-_\sigma(s_i) = \left\{
     \begin{array}{lr}
       |\{j < i : s_j = c \}|  & :  s_i = a\\
 |\{j < i : s_j = a \}|  & :  s_i = b\\     
  |\{j < i : s_j = b \}|  & :  s_i = c 
 \end{array}
   \right.
  \end{displaymath}


For example, if $s_{i} = a$, then $q^+(s_i)$ is the number of sides of die $B$ whose labels precede $i$.  Similarly, $q^-(s_i)$ is the number of sides of die $C$ whose labels precede $i$.

\begin{example}\label{3sided}
{\rm Let $D$ be the following set of dice.
\begin{align*}
A &= 9 \hspace{5pt} 5 \hspace{5pt} 1 \\
B & = 8 \hspace{5pt} 4 \hspace{5pt} 3 \\
C &= 7 \hspace{5pt} 6 \hspace{5pt} 2
\end{align*}
Then $\sigma(D) = acbbaccba$.  Note that this set of dice is balanced and non-transitive, as $P(A \succ B) = P(B \succ C) = P(C \succ A) = \frac{5}{9}$. }
\end{example}

Conversely, given a word $\sigma$, let $D(\sigma)$ denote the unique set of dice corresponding to $\sigma$.  As this is a one-to-one correspondence, we often speak of a set of dice and the associated word interchangeably.  For instance, if $\sigma = s_1 s_2 \cdots s_{3n}$ is a $3n$-letter word, the probability of die $A$ beating die $B$ is given by
\[
P(A \succ B) = \frac{1}{n^2} \sum_{s_i = a} q^+(s_i),
\]
and the other probabilities may be computed analogously.  Thus, the property of a set $D$ of dice being balanced is equivalent to $\sigma(D)$ satisfying
\[
\sum_{s_i = a} q^+(s_i) = \sum_{s_i = b} q^+(s_i) = \sum_{s_i = c} q^+(s_i).
\]
Furthermore, if $D$ is a set of $n$-sided dice, then $D$ is non-transitive if and only if each of
\[
\sum_{s_i = a} q_{\sigma(D)}^+(s_i),  \sum_{s_i = b} q_{\sigma(D)}^+(s_i), \text{ and } \sum_{s_i = c} q_{\sigma(D)}^+(s_i)
\]
exceeds $n^2/2$.  

Although a set of dice $D$ and its associated word $\sigma(D)$ hold the same information, this alternate interpretation will prove invaluable in showing Theorem \ref{thm}.  First, we need some lemmas.  Recall that the \emph{concatenation} of two words $\sigma$ and $\tau$, for which we write $\sigma\tau$, is simply the word $\sigma$ followed by $\tau$. 

\begin{lemma}\label{first}
Let $\sigma$ and $\tau$ be balanced words.  Then the concatenation $\sigma \tau$ is balanced. 
\end{lemma}

\begin{proof}
Let $|\sigma|=3m$, $|\tau|=3n$. If $i\leq3m$, then
$q^{+}_{\sigma\tau}(s_{i})=q^{+}_{\sigma}(s_{i})$ ($q^{+}$ is defined as a subset of the $s_j$ with $j<i$, so concatenating $\tau$ after $\sigma$ contributes nothing to these).
Otherwise (for $3m<i\leq 3m+3n$), $q^{+}_{\sigma\tau}(s_{i})=q^{+}_{\tau}(s_{i})+m$, because every letter from $\tau$ beats all $m$ letters from the appropriate die in $\sigma$, in addition to whichever letters it beats from the structure of $\tau$ itself.
Then 
\begin{align}\label{concat}
\sum_{s_{i}=a}q^{+}_{\sigma\tau}(s_{i})=\sum_{s_{i}=a}q^{+}_{\sigma}(s_{i})+\sum_{s_{i}=a}q^{+}_{\tau}(s_{i})+mn.
\end{align}
We may repeat the argument for $s_{i}=b,c$, and then we are done as $\sigma$ and $\tau$ are balanced.
\end{proof}

While Lemma \ref{first} is primarily useful for balanced words (or sets of dice), the next result applies to arbitrary sets of non-transitive dice.  

\begin{lemma}\label{concatbalanced}
Let $\sigma$ and $\tau$ be non-transitive words.  Then the concatenation $\sigma\tau$ is non-transitive.  
\end{lemma}

\begin{proof}
Let $\sigma$ be a word of length $3m$.  Because $m^2 P_\sigma( A \succ B)$ counts the number of rolls of dice $A$ and $B$ in which die $A$ beats die $B$, we note that
\[
m^2 P_\sigma( A \succ B) = \sum_{s_i = a} q^+(s_i),
\]
and an analogous statement holds for $m^2 P_\sigma( B \succ C)$ and $m^2 P_\sigma( C \succ A)$.  Define a quantity $V_\sigma$ by 
\[
V_\sigma = m^2 \cdot \min \{ P_\sigma( A \succ B) , P_\sigma( B \succ C), P_\sigma( C \succ A) \}. 
\] 

Now let $\tau$ be a word of length $3n$, and define quantities $V_\tau$ and $V_{\sigma \tau}$ as above.  Note that 
\[
V_{\sigma}>\frac{m^{2}}{2},\;\;V_{\tau}>\frac{n^{2}}{2},
\] 
because $\sigma$ and $\tau$ are non-transitive.  By Equation \ref{concat}, we have 
\begin{align*}
V_{\sigma\tau} & = V_{\sigma}+V_{\tau}+mn\\
& > \frac{m^{2}}{2}+\frac{n^{2}}{2}+mn\\
& = \frac{(m+n)^{2}}{2},\\
\end{align*}
and so $\sigma\tau$ is non-transitive.
\end{proof}

With the two lemmas above in place, we are now able to provide a quick proof of Theorem \ref{thm}, the main result of this section.

\begin{proof}[Proof of Theorem \ref{thm}]
Example \ref{3sided}, along with
\begin{example}\label{4sided}
$$\begin{array}{cccccc}A & = & 12 & 10 & 3 & 1 \\B & = & 9 & 8 & 7 & 2 \\C & = & 11 & 6 & 5 & 4\end{array}$$
\end{example}

\noindent
and

\begin{example}\label{5sided}
$$\begin{array}{ccccccc}A & = & 15 & 11 & 7 & 4 & 3 \\B & = & 14 & 10 & 9 & 5 & 2 \\C & = & 13 & 12 & 8 & 6 & 1\end{array}$$
\end{example}

\noindent
provide balanced, non-transitive sets of dice for $n=3,4,5$, which give rise to balanced words for these $n$, the smallest representatives (in the context of the theorem) for each congruence class modulo $3$. Lemmas \ref{first} and \ref{concatbalanced} then imply that the concatenation of two balanced non-transitive words is a balanced non-transitive word, and the correspondence between words and sets of dice completes the proof.
\end{proof}

\ssection{Face-sums}

After taking a closer look at Example \ref{3sided} as well as the sets of balanced, non-transitive dice given in the proof of Theorem \ref{thm}, one may notice the following phenomenon: In any one of these sets of dice, the sum of the labels of any two dice are equal.  Formally, if $D$ is a set of $n$-sided dice and $\sigma(D) = s_1s_2 \cdots s_{3n}$, we define the \emph{face-sums} of $D$ to be
\[
\sum_{s_i = a} i, \sum_{s_i = b} i , \text{ and } \sum_{s_i = c} i.
\]

\begin{theorem}\label{facesum}
A set of dice $D$ (or the corresponding word) is balanced if and only if the face-sums of its dice are all equal. 
\end{theorem}

\begin{proof}
(\emph{only if}) Let $D$ be a set of balanced dice, and $\sigma(D)$ the word associated with it. Recall our definition for balanced words:
\[
\sum_{s_{i}=a}q^{+}_{\sigma}(s_{i})=\sum_{s_{i}=b}q^{+}_{\sigma}(s_{i})=\sum_{s_{i}=c}q^{+}_{\sigma}(s_{i}),
\]
which is obviously equivalent to
\[
\sum_{s_{i}=a}q^{-}_{\sigma}(s_{i})=\sum_{s_{i}=b}q^{-}_{\sigma}(s_{i})=\sum_{s_{i}=c}q^{-}_{\sigma}(s_{i}).
\]
Further define 
\[
q_{\sigma}(s_{i})=|\{j<i: s_{j}=s_{i}\}|.
\]
We focus on die $A$ (with face-sum $\sum_{s_{i}=a}i$), and make two observations: 
First, for a face of $A$, its label $i$ may be written as 
\[
i=q^{+}_{\sigma}(s_{i})+q^{-}_{\sigma}(s_{i})+q_{\sigma}(s_{i})+1.
\]
Second, since $A$ has $n$ sides, 
\[
\sum_{s_{i}=a}q_{\sigma}(s_{i})=\frac{n(n-1)}{2}.
\]
Then, $$\sum_{s_{i}=a}i=\sum_{s_{i}=a}\left(q^{+}_{\sigma}(s_{i})+q^{-}_{\sigma}(s_{i})+q_{\sigma}(s_{i})+1\right)=\sum_{s_{i}=a}q^{+}_{\sigma}(s_{i})+\sum_{s_{i}=a}q^{-}_{\sigma}(s_{i})+\frac{n(n-1)}{2}+n.$$
However, this computation was independent of our choice of $A$, so the other two are analogous, and every term in sight is equal as $\sigma(D)$ is balanced. \\

(\emph{if}) Let $D$ be a set of $n$-sided dice (with word $\sigma(D)$), and assume that
$$\sum_{s_{i}=a}i=\sum_{s_{i}=b}i=\sum_{s_{i}=c}i.$$
By the above, this is equivalent to $$\sum_{s_{i}=a}q^{+}_{\sigma}(s_{i})+\sum_{s_{i}=a}q^{-}_{\sigma}(s_{i})=\sum_{s_{i}=b}q^{+}_{\sigma}(s_{i})+\sum_{s_{i}=b}q^{-}_{\sigma}(s_{i})=\sum_{s_{i}=c}q^{+}_{\sigma}(s_{i})+\sum_{s_{i}=c}q^{-}_{\sigma}(s_{i}).$$
Write $$a^{+}=\sum_{s_{i}=a}q^{+}_{\sigma}(s_{i}),\;\;a^{-}=\sum_{s_{i}=a}q^{-}_{\sigma}(s_{i})$$ and analogously define $$b^{+},b^{-},c^{+}, \text{ and }c^{-}.$$
Then, we have $$a^{+}+a^{-}=b^{+}+b^{-}=c^{+}+c^{-},$$
and $$a^{+}+b^{-}=b^{+}+c^{-}=c^{+}+a^{-}(=n^{2}),$$ giving six equations in six unknowns.
Some straightforward linear algebra gives $$a^{+}=b^{+}=c^{+},$$ whence we also have $$a^{-}=b^{-}=c^{-}.$$
\end{proof}

Applying the result of Theorem \ref{facesum}, we obtain the following algorithm for checking if a given partition of $[3n]$ into $3$ size-$n$ subsets determines a set of balanced non-transitive dice.

\begin{algorithm}
{\rm Suppose we are given a partition of $[3n]$ into $3$ size $n$ subsets $A,B,$ and $C$.  First, check the sums of the elements of these subsets.  These sums are equal if and only if the set of dice is balanced (by Theorem \ref{facesum}).  If this condition is met, check $P(A \succ B)$.  If $P(A \succ B) = 1/2$, the set of dice is balanced but fair.  If $P(A \succ B) > 1/2$, the set is balanced and non-transitive.  If $P(A \succ B) < 1/2$, switching the labels of sets $B$ and $C$ produces a balanced non-transitive set of dice.  Since this algorithm must check each pair of sides from dice $A$ and $B$, it clearly runs in $O(n^2)$ time.  }
\end{algorithm}

By contrast, not checking the face-sums, using only the probabilities to check balance would take roughly 3 times as long.

\ssection{Other constructions}

\ssubsection{Non-transitive dice and Fibonacci numbers}

In \cite{ntD1}, Savage forms sets of non-transitive dice from consecutive terms of the Fibonacci Sequence.  We briefly explain his construction.  (We let $f_i$ denote the $i^\text{th}$ Fibonacci number, so that $f_1 = f_2 =1, f_3 = 2,$ et cetera.) 

\begin{algorithm}[\cite{ntD1}]\label{savage}
{\rm Given a Fibonacci number \emph{$f_{k}$}, consider the sequence 
\[
f_{k-2},f_{k-1},f_{k},f_{k-1},f_{k-2}.
\] 
Beginning with the number $3f_{k}$, label die $A$ with $f_{k-2}$ consecutive descending integers. Then label die $B$ with the next $f_{k-1}$ values, die $C$ with the next $f_{k}$ values, $A$ with the next $f_{k-1}$ values, and $B$ with the last (ending in $1$) $f_{k-2}$ values. This produces a set of non-transitive dice (which is never balanced).}
\end{algorithm}

In the case where $f_{k}$ is an \emph{odd} Fibonacci number, a simple addition to this construction actually yields a balanced set.

\begin{algorithm}
{\rm Perform Algorithm \ref{savage} to obtain a set of non-transitive dice. Then, swap the last element of the first set of values ($3f_{k}-f_{k-2}+1$, to be precise), which is on die $A$, with the first element of the second set of values ($3f_{k}-f_{k-2}$), which is the largest number on die $B$. The resulting set of dice is non-transitive and balanced.}
\end{algorithm}

\ssubsection{Sets of four dice}

Consider a modification of \emph{set of dice} to mean four dice, labeled $A,B,C,D$. Then

\begin{example}\label{3sided4}
$$\begin{array}{cccc}A: & 12 & 5 & 2 \\B: & 11 & 8 & 1 \\C: & 10 & 7 & 3 \\D: & 9 & 6 & 4 \end{array}$$
\end{example}

\begin{example}\label{4sided4}
$$\begin{array}{ccccc}A: & 16 & 10 & 7 & 1 \\B: & 15 & 9 & 6 & 4 \\C: & 14 & 12 & 5 & 3 \\D: & 13 & 11 & 8 & 2 \end{array}$$
\end{example}

\begin{example}\label{5sided4}
$$\begin{array}{cccccc}A: & 20 & 13 & 10 & 6 & 4 \\B: & 19 & 15 & 9 & 8 & 3 \\C: & 18 & 16 & 12 & 5 & 1 \\D: & 17 & 14 & 11 & 7 & 2 \end{array}$$
\end{example}

\noindent give minimal examples for balanced non-transitive sets of dice. Modifying to $4n$-letter words (with $n$ each of $a,b,c,d$), the proof of Theorem \ref{thm} generalizes, which gives us the following.

\begin{theorem}\label{mainfour}
For any $n \geq 3$, there exists a set of four balanced, non-transitive, $n$-sided dice.  
\end{theorem}

However, Example \ref{3sided4} has unequal face-sums, proving that Theorem \ref{facesum} does not generalize.

\ssection{Further Questions}

Given the proof of Theorem \ref{thm}, it seems natural to define the following.  

\begin{definition}
{\rm Let $\sigma$ be a balanced non-transitive word.  If there do not exist balanced non-transitive words $\tau_1$ and $\tau_2$ (both nonempty) such that $\sigma = \tau_1 \tau_2$, we say that $\sigma$ (and its associated set of dice) is \emph{irreducible}.  }
\end{definition}

\begin{question}
For any $n \geq 3$, does there necessarily exist an irreducible balanced non-transitive set of $n$-sided dice?
\end{question}

The notion of non-transitive triples of dice also suggests the following broad generalization.  

\begin{definition}
{\rm Let $G$ be an orientation of $K_m$, the complete graph on the vertex set $\{v_1, v_2, \ldots, v_m\}$.  Define a \emph{realization} of $G$ to be an $m$-tuple of $n$-sided dice $A_1, A_2, \ldots, A_m$ for some $n$ (where now the $A_i$'s partition $[mn]$) satisfying the following property:
\[
P(A_i \succ A_j) > \frac{1}{2} \Leftrightarrow (v_i \rightarrow v_j) \text{ is an edge of } G.
\] }
\end{definition}

Theorem \ref{thm} gives us the following as a corollary.   

\begin{corollary}
Let $G$ be an orientation of $K_3$.  Then there exists a realization of $G$ using $n$-sided dice for any $n \geq 3$.  
\end{corollary}

\begin{proof}
If $G$ is a directed cycle, Theorem \ref{thm} gives the result.  Otherwise, $G$ is acyclic, meaning the orientation corresponds to a total ordering of the vertices.  Then the dice $A = \{1, 2, \ldots, n\}, B= \{ n+1, n+2, \ldots, 2n\},$ and $C = \{2n+1, 2n+2, \ldots, 3n\}$, appropriately placed, will provide a realization. 
\end{proof}

\begin{question}
Given an orientation of $K_m$, can one always find a set of $n$-sided dice (for some $n$) which realizes this orientation?
\end{question}

\end{document}